\documentclass[a4paper,11pt,twoside]{article}
\usepackage{amsmath,amsfonts,amssymb,amsthm}

\usepackage {epsfig}

\newcommand{\Real}{\mathbb{R}}

\newtheorem{theorem}{Theorem}

\newtheorem {lemma}{Lemma}

\title{A quantitative discounted central limit theorem using the Fourier metric}
\author{Guy Katriel\\ Department of Mathematics, ORT Braude College,\\ Karmiel, Israel\\}

\date{}

\begin{document}

\maketitle

\begin{abstract}
The discounted central limit theorem concerns the convergence of an infinite discounted sum of i.i.d. random variables
to normality as the discount factor approaches $1$. We show that, using the Fourier metric on probability distributions,
one can obtain the discounted central limit theorem, as well as a quantitative version of it, in a simple and natural way, and under weak 
assumptions.
\end{abstract}

\section{Introduction}

Let $X_n$ ($n\geq 0$) be a sequence of i.i.d. real-valued random variables, with 
\begin{equation}\label{ba}\mu=E(X_0),\;\;\; \sigma^2=Var(X_0)<\infty.\end{equation}
For $a\in [0,1)$, we define the random variable
\begin{equation}\label{rs}S_a=\sum_{n=0}^\infty a^n X_n.\end{equation}
Standard results ensure that (\ref{rs}) converges almost surely (see e.g. \cite{chung}, Sec. 5.3).
$S_a$ can be understood as the present value of a future stream of i.i.d. payments, where $a$ is the discount factor.

Gerber \cite{gerber} proved, assuming that $X_n$ have finite third moments, that as $a\rightarrow 1$, the distribution of $S_a$ 
approaches a normal distribution: normalizing $S_a$ by setting
\begin{equation}\label{nz}\hat{S}_a=\frac{S_a-E(S_a)}{\sqrt{Var(S_a)}}=\frac{\sqrt{1-a^2}}{\sigma}\cdot \left(S_a-\frac{\mu}{1-a}\right),\end{equation}
we have
\begin{equation}\label{clt}\hat{S}_a\;\;\substack{D\\ \rightarrow} \;\;N(0,1)\;\;\;\;{\mbox{as }}\;a\rightarrow 1-,\end{equation}
that is, defining the corresponding cumulative distribution functions
\begin{equation}
\label{dfa}F_a(x)=P\left(\hat{S}_a\leq x \right),
\end{equation}
\begin{equation}
\label{dn}\Phi (x)= P\left(N(0,1)\leq x \right)=\frac{1}{\sqrt{2\pi}} \int_{-\infty}^x e^{-\frac{u^2}{2}} du,
\end{equation}
we have, for all $x\in \Real$,
\begin{equation}\label{cid}\lim_{a\rightarrow 1-}F_a(x)=\Phi(x).\end{equation}
This is the discounted central limit theorem.
Gerber also gave a quantitative bound of Berry-Eseen type for this convergence: 
\begin{equation}\label{gerber}\sup_{x\in\Real} |F_a(x)-\Phi(x)|\leq C\cdot \frac{E(|X_0-\mu|^3)}{\sigma^3}\cdot (1-a)^{\frac{1}{2}},\end{equation}
and one can take $C=5.4$ (we note that the formulation given in \cite{gerber} is 
slightly different, but equivalent, because of the different normalization taken there).
Subsequent works extended and refined the results of \cite{gerber} in several directions (see e.g. \cite{embrechts,horvath, saulis,whitt}).

Here we will prove a discounted central limit theorem without any assumption on moments higher than $2$, and also
give a new and different quantitative bound for the convergence, in the case that some moment of order $s=2+\epsilon$ ($\epsilon>0$) exists.
This bound will be given in terms of a Fourier-based metric, which will be seen to provide a simple and natural approach to the study of 
discounted sums.
A key observation underlying our proofs is that the
distibution $F_a$ can be realized as a fixed point of a mapping on a space of distributions, which is a contraction with respect to
this metric. 
Fourier-based metrics were introduced in connection with study of the Boltzmann equation \cite{gabetta}, and have since found many applications (see \cite{carrillo} for a review). 
In particular in \cite{goudon} these metrics have been used to obtain Berry-Esseen type inequalities.

We recall the definition of the Fourier-based metrics.  
For any real $s>0$, we denote by ${\cal{P}}^s$ the set of all distribution functions $G$ on $\Real$ with finite moment of order $s$, and with expectation $0$ and variance $1$:
\begin{equation}\label{fs}\int_{-\infty}^\infty |x|^s dG(x)<\infty,\end{equation}
\begin{equation}
\label{mc}\int_{-\infty}^\infty x dG(x)=0,\;\;\; \int_{-\infty}^\infty x^2 dG(x)=1.
\end{equation}
To each distribution function $G$  we associate its characteristic function
$$C_G(\xi)=\int_{-\infty}^\infty e^{-i\xi x} dG(x).$$  
If $s>0$, and $G,H$ are probability distributions, their Fourier distance of type $s$ is defined by
\begin{equation}
\label{fd}d_s(G,H)=\sup_{\xi\neq 0} \frac{|C_G(\xi)-C_H(\xi)|}{|\xi|^s}.
\end{equation}
If $s\in [2,3]$ and $G,H\in {\cal{P}}^s$, then $d_s(G,H)<\infty$ (see \cite{carrillo}, Proposition 2.6).

We prove that
\begin{theorem}\label{main2} If (\ref{ba}) holds, then
\begin{equation}\label{mi} \lim_{a\rightarrow 1-}d_2(F_a,\Phi)=0.\end{equation}
\end{theorem}
(\ref{mi}) implies pointwise convergence of $C_{F_a}$ to $C_{\Phi}$, which, by
Levy's Continuity Theorem (see e.g. \cite{chung}, Sec 6.3), implies (\ref{cid}). The validity of the discounted central limit theorem
(\ref{clt}), without any assumption on moments higher than $2$, thus follows from Theorem \ref{main2} - we note however that it also follows from previous results such as those in \cite{embrechts}.

A quantitative version of Theorem \ref{main2} can be obtained if we assume that $X_n$ have a finite
$s$-moment for some $s>2$. Set
$$\hat{X}_0=\sigma^{-1}(X_0-\mu),$$
and let $F$ denote its distribution function:
\begin{equation}
\label{dpsi}F(x)=P\left(\hat{X}_0\leq x \right).
\end{equation}
Note that by (\ref{ba}) have $F\in  {\cal{P}}^2$.
\begin{theorem}\label{main3} Assume $F\in  {\cal{P}}^s$ where $s\in (2,3]$. Then, for $a\in (0,1)$
\begin{equation}\label{nn}d_2(F_a,\Phi)\leq  \left[\frac{(s-2)(1-a^2)}{e\cdot a^2 }\right]^{\frac{1}{2}(s-2)}\cdot d_s(F,\Phi).\end{equation}
\end{theorem} 

Note that $s>2$ implies that the right-hand side of (\ref{nn}) goes to $0$ as $a\rightarrow 1$, so that (\ref{nn}) implies (\ref{mi}) for $s>2$ (but not for $s=2$, which
is the reason that Theorem \ref{main2} needs a separate proof).

Comparing the bound of Theorem \ref{main3} with Gerber's bound (\ref{gerber}), we note several differences.

(1) Theorem $\ref{main3}$ provides a bound whenever some moment of order $s>2$ is finite, while (\ref{gerber}) 
requires a finite third central moment for $X_0$.

(2) The bound (\ref{gerber}) is universal, hence does not take into account the distance between the distribution 
of $X_0$ and the normal distribution. In (\ref{nn}), the bound becomes small if $X_0$ is close to normal.

(3) A major difference is of course the fact that the distance between distributions is measured differently: while 
(\ref{nn}) uses a Fourier metric, (\ref{gerber}) uses the Kolmogorov metric. In fact it is possible to bound the Kolmogorov metric
in terms of the $d_2$ metric: using the Berry-Eseen inequality (see \cite{feller}, Sec. XVI.4, Lemma 2) we get
$$|F_a(x)-\Phi(x)|\leq \frac{1}{\pi}\int_{-T}^T \frac{|C_{F_a}(\xi)-C_{\Phi}(\xi)|}{|\xi|^2} |\xi|d\xi +\frac{24}{\pi T}$$
$$\leq \frac{1}{\pi}\cdot d_2(F_a,\Phi)\int_{-T}^T |\xi|d\xi +\frac{24}{\pi T}= \frac{1}{\pi}\cdot d_2(F_a,\Phi)T^2 +\frac{24}{\pi T},$$
and optimizing over $T$ gives
$$\sup_{x\in\Real} |F_a(x)-\Phi (x)|\leq \frac{3\cdot 12^{\frac{2}{3}}}{\pi}\cdot \left(d_2(F_a,\Phi) \right)^{\frac{1}{3}},$$
so that convergence in the $d_2$ metric implies convergence in the Kolmogorov metric (as well as in the Wasserstein metric, see
\cite{carrillo}, Theorem 2.21). However, it should be noted that using this bound together with (\ref{nn}) gives a
bound of order $O((1-a)^{\frac{1}{6}(s-2)})$ as $a\rightarrow 1$ for the convergence of the Kolmogorov metric, which in the case 
$s=3$ (which is relevant for this comparison) gives $O((1-a)^{\frac{1}{6}})$, a weaker convergence rate than
the one given by (\ref{gerber}). 

We thus conclude that none of the inequalities (\ref{gerber}) and (\ref{nn}) is a consequence of the other, and each has its 
advantages. It might be an interesting problem to obtain bounds which combine the advanatges of the two inequalities.

\section{Proofs of the theorems}

Noting that $\hat{S}_a$ does not change if a linear function is applied to all $X_n$'s, there is no loss of generality in proving our results 
under the normalization
$$E(X_n)=0,\;\;\;Var(X_n)=1,$$
which we will henceforth assume. Under this assumption the distribution $F$ given by (\ref{dpsi}) is simply the distribution of the $X_n$'s:
\begin{equation}
\label{dpsi1}F(x)=P\left(X_0\leq x \right),
\end{equation}
and (\ref{nz}) becomes
\begin{equation}\label{nz1}\hat{S}_a=\sqrt{1-a^2}\cdot S_a.\end{equation}


For $a\in [0,1)$, we now define the operator $T_a: {\cal{P}}^2\rightarrow  {\cal{P}}^2$, whose unique fixed point will later be shown to be the distribution function $F_a$ of $\hat{S}_a$. 
If $G\in  {\cal{P}}^2$, let $Y$ be a random variable with distribution function $G$. Let $X$ be a random variable independent of $Y$, with
distribution function $F$ given by (\ref{dpsi1}). Then $T_a[G]$ is defined to be the distribution function of $a Y +\sqrt{1-a^2}\cdot X$:
$$T_a[G](x)=P\left(a Y +\sqrt{1-a^2}\cdot X \leq x \right).$$
Since
$$E\left(a Y +\sqrt{1-a^2}\cdot X\right)= aE(Y)+\sqrt{1-a^2}\cdot E(X)=0,$$
$$Var\left(a Y +\sqrt{1-a^2}\cdot X\right)= a^2Var(Y)+(1-a^2)Var(X)=a^2 +(1-a^2)=1,$$
we indeed have $T_a[G]\in {\cal{P}}^2$.

By the above definition, the $n$-fold composition $T_a^n[G]$ ($n\geq 1$) is the distribution function of the random variable $Y_n$  
given by the autoregressive process
\begin{equation}\label{ar}Y_{n+1}=a Y_n + \sqrt{1-a^2}\cdot X_n,\;\;n\geq 0
\end{equation}
when $G$ is the distribution function of $Y_0$.

In terms of characteristic functions we have
\begin{equation}
\label{dta}C_{T_a[G]}(\xi)=C_{F}\left(\sqrt{1-a^2}\cdot \xi\right)\cdot C_G(a\xi).
\end{equation}

The following Lemma summarizes properties of the operator $T_a$.

\begin{lemma}\label{fp} For $a\in [0,1)$:  
	\begin{itemize}
\item[(i)] If $G,H\in  {\cal{P}}^2$ then For any integer $n\geq 1$
\begin{equation}\label{cont}d_2(T_a^n[G],T_a^n[H])\leq a^{2n}\cdot d_2(G,H).\end{equation}
	
\item[(ii)]	$F_a$, given by (\ref{dfa}), is a fixed point of $T_a$, and if $a>0$ it is the unique fixed point.

\item[(iii)] For any $G\in  {\cal{P}}^2$ we have
\begin{equation}\label{bob}d_2\left( G,F_a \right)\leq \frac{d_2\left(G,T_a[G]\right)}{1-a^2}.\end{equation}
\end{itemize}
\end{lemma}

\begin{proof} (i) For $n=1$ we have, using (\ref{dta}) and the fact that $|C_{F}(\xi)|\leq 1$,
		$$d_2(T_a[G],T_a[H])= \sup_{\xi\neq 0} \frac{|C_{F}(\sqrt{1-a^2}\cdot \xi)| \cdot |C_G(a\xi)-C_H(a\xi)|}{\xi^2}$$
	$$\leq  \sup_{\xi\neq 0} \frac{ |C_G(a\xi)-C_H(a\xi)|}{\xi^2}=a^2\cdot  \sup_{\xi\neq 0} \frac{ |C_G(a\xi)-C_H(a\xi)|}{(a\xi)^2}=a^2\cdot d_2(G,H).$$
	Proceeding by induction, we have
	$$d_2\left(T_a^{n+1}[G],T_a^{n+1}[H]\right)\leq a^2\cdot d_2\left(T_a^n[G],T_a^n[H]\right)=a^{2(n+1)}\cdot d_2(G,H).$$
	
	 (ii) We denote equality in distribution of two random variables by $\substack{D\\=}$. Let $X$ be a random variable with distribution $F$, independent of $S_a$.
	 We claim that
	 \begin{equation}
	 \label{yy}a\hat{S}_a+\sqrt{1-a^2}\cdot X\;\;\substack{D\\=}\;\;\hat{S}_a,
	 \end{equation}
	 which implies $T_a[F_a]=F_a.$
	 To show (\ref{yy}), note that we have
	 $$S_a=\sum_{n=0}^\infty a^n X_n = X_0+ a\sum_{n=0}^\infty a^n X_{n+1}\;\;\substack{D\\=}\;\;X+aS_a,$$
	 so that, using (\ref{nz1}),
	 $$a\hat{S}_a+\sqrt{1-a^2}\cdot X=\sqrt{1-a^2}\cdot \left[aS_a +X\right]\;	\substack{D\\=}\;\sqrt{1-a^2}\cdot S_a=\hat{S}_a,$$
	 so we have (\ref{yy}).
	 
To show uniqueness of the fixed point when $a>0$, assume $G\in {\cal{P}}^2$, $T_a[G]=G$. Using (\ref{cont}),
$$d_2(F_a,G)=d_2(T[F_a],T[G])\leq a^2\cdot d_2(F_a,G)\;\;\Rightarrow\;\; d_2(F_a,G)=0\;\;\Rightarrow\;\; G=F_a.$$

(iii) By (\ref{cont}) we have
$$d_2\left(T_a^{k}[G],T_a^{k+1}[G]\right)\leq a^{2k} d_2\left(G,T_a[G]\right),$$
so the triangle inequality gives
\begin{equation}
\label{ini}
d_2\left(G,T_a^{n}[G] \right)\leq \sum_{k=0}^{n-1} d_2\left(T_a^{k}[G],T_a^{k+1}[G]\right)\leq  \frac{1-a^{2n}}{1-a^2}\cdot d_2\left(G,T_a[G]\right).
\end{equation}
From (i),(ii) we have
\begin{equation}\label{conv}d_2(T_a^n[G],F_a)=d_2(T_a^n[G],T_a^n[F_a])\leq a^{2n}\cdot d_2(G,F_a).\end{equation}
Therefore, using the triangle inequality and (\ref{ini}),(\ref{conv}),
\begin{eqnarray*}
	d_2\left(G,F_a \right)&\leq& d_2\left(G, T_a^n[G] \right) + d_2\left( T_a^n[G],F_a \right)\nonumber\\ 
	&\leq& \frac{1-a^{2n}}{1-a^2}\cdot d_2\left(G,T_a[G]\right)+a^{2n}\cdot d_2\left(G,F_a \right),
\end{eqnarray*}
and taking $n\rightarrow \infty$ we obtain (\ref{bob}).
\end{proof} 

The following Lemma plays a key role in proving the theorems:

\begin{lemma}\label{ncomp} Let $\Phi$ be the Normal distribution function (\ref{dn}). Then for $a\in [0,1)$
$$d_2\left(F_a,\Phi \right)\leq \sup_{w\neq 0} \left[e^{-\frac{(a w)^2}{2(1-a^2)}}\cdot\frac{|C_{F}( w)-C_{\Phi}(w)|}{w^2}\right].$$
\end{lemma}

\begin{proof} 
	Applying (\ref{bob}) with $G=\Phi$ we have
\begin{equation}
\label{lkl}d_2\left(\Phi,F_a \right)\leq \frac{d_2\left(\Phi,T_a[\Phi]\right)}{1-a^2}.
\end{equation}
The characteristic function of the normal distribution  $\Phi$ is given by
$C_{\Phi}(\xi)=e^{-\frac{1}{2}\xi^2}$,
so using (\ref{dta}) and the substitution $w=\sqrt{1-a^2}\cdot \xi$, we have
\begin{eqnarray}\label{as}
&&\frac{d_2(T_a[\Phi],\Phi)}{1-a^2}=\sup_{\xi\neq 0} \frac{|C_{T_a[\Phi]}(\xi)-C_{\Phi}(\xi)|}{(1-a^2)\xi^2}\\
&=&\sup_{\xi\neq 0} \frac{|C_{F}( \sqrt{1-a^2}\cdot \xi)e^{-\frac{1}{2}(a\xi)^2}-e^{-\frac{1}{2}\xi^2}|}{(1-a^2)\xi^2}\nonumber\\
&=&\sup_{\xi\neq 0} \left[e^{-\frac{(a\xi)^2}{2}}\cdot\frac{|C_{F}(\sqrt{1-a^2}\cdot \xi)-e^{-\frac{1}{2}(1-a^2)\xi^2}|}{(1-a^2)\xi^2}\right]\nonumber\\
&=&\sup_{w\neq 0} \left[ e^{-\frac{(a w)^2}{2(1-a^2)}}\cdot\frac{|C_{F}(w)-e^{-\frac{1 }{2}w^2}|}{w^2}\right]=\sup_{w\neq 0} \left[e^{-\frac{(a w)^2}{2(1-a^2)}}\cdot\frac{|C_{F}(w)-C_{\Phi}(w)|}{w^2}\right].\nonumber
\end{eqnarray}
Combining (\ref{lkl}) and (\ref{as}) we have the result.
\end{proof}

We can now give the proofs of the theorems.

\begin{proof}[Proof of Theorem \ref{main2}]
By Lemma \ref{ncomp} it suffices to show that
$$\lim_{a\rightarrow 1}\sup_{w\neq 0} \left[e^{-\frac{(a w)^2}{2(1-a^2)}}\cdot\frac{|C_{F}( w)-C_{\Phi}(w)|}{w^2}\right]=0.
$$
Fix $\epsilon>0$. 
By the assumption (\ref{ba}), $C_F$ and $C_{\Phi}$ are twice differentiable, with 
$$C_F(0)=C_{\Phi}(0)=1,\;\;C_F'(0)=C_{\Phi}'(0)=0,\;\;C_F''(0)=C_{\Phi}''(0)=-1$$
so application of L'Hospital's rule gives
$$\lim_{w\rightarrow 0}\frac{C_{F}(w)-C_{\Phi}(w)}{w^2}=0.$$
Therefore we can choose $\delta>0$ so that
\begin{equation}\label{p1}|w|< \delta \;\;\Rightarrow\;\; e^{-\frac{(aw)^2}{2(1-a^2)}}\cdot\frac{|C_{F}(w)-C_{\Phi}(w)|}{w^2}\leq \frac{|C_{F}(w)-C_{\Phi}(w)|}{w^2}< \epsilon. \end{equation}
Using the fact that $|C_F(w)|,|C_{\Phi}(w)|\leq 1$, we have
$$|w|\geq \delta \;\;\Rightarrow\;\;e^{-\frac{(aw)^2}{2(1-a^2)}}\cdot\frac{|C_{F}(w)-C_{\Phi}(w)|}{w^2}\leq \frac{2}{w^2}e^{-\frac{(aw)^2}{2(1-a^2)}}\leq  \frac{2}{\delta^2}e^{-\frac{(a\delta)^2}{2(1-a^2)}}.$$
The right-hand side of the above inequality goes to $0$ as $a\rightarrow 1$, hence for $a$ sufficiently close to $1$ we have
\begin{equation}\label{p2} |w|\geq \delta \;\;\Rightarrow\;\;e^{-\frac{(aw)^2}{2(1-a^2)}}\cdot\frac{|C_{F}(w)-C_{\Phi}(w)|}{w^2}<\epsilon.\end{equation}
From (\ref{p1}),(\ref{p2}) we have that, for $a$ sufficiently close to $1$,
$$\sup_{w\neq 0} \left[e^{-\frac{(a w)^2}{2(1-a^2)}}\cdot\frac{|C_{F}( w)-C_{\Phi}(w)|}{w^2}\right]<\epsilon,$$
 concluding the proof.
\end{proof}

\begin{proof}[Proof of Theorem \ref{main3}] Assume $s>2$. Using Lemma \ref{ncomp} we have 
\begin{eqnarray}\label{rw}d_2\left(F_a,\Phi \right)&\leq&\sup_{w\neq 0} \left[e^{-\frac{(aw)^2}{2(1-a^2)}}\cdot\frac{|C_{F}(w)-C_{\Phi}(w)|}{w^2}\right]\nonumber\\
&=&\sup_{w\neq 0} \left[e^{-\frac{(aw)^2}{2(1-a^2)}}\cdot |w|^{s-2}\cdot\frac{|C_{F}(w)-C_{\Phi}(w)|}{|w|^s}\right]\nonumber\\
&\leq& d_s(F,\Phi)\cdot \sup_{w\neq 0} \left[e^{-\frac{(aw)^2}{2(1-a^2)}}|w|^{s-2}\right]\end{eqnarray}
Computing the supremum on the right-hand side of (\ref{rw}) by elementary calculus we find that it is attained at $w=\pm\frac{\sqrt{(s-2)(1-a^2)}}{a}$, hence
$$\sup_{w\neq 0} \left[e^{-\frac{(aw)^2}{2(1-a^2)}}|w|^{s-2}\right]=\left[\frac{ (s-2)(1-a^2)}{e\cdot a^2 }\right]^{\frac{1}{2}(s-2)},$$
which gives (\ref{nn}).
\end{proof}

\end{document}